\documentclass[11pt,reqno]{article}
\usepackage{amssymb, amsmath, amsthm, amsfonts, amscd, epsfig, subcaption}
\usepackage[dvipsnames]{xcolor}
\usepackage[utf8]{inputenc}
\usepackage[english]{babel}
\usepackage{bm}
\usepackage{bbm}
\usepackage{graphicx, epsfig}
\usepackage{geometry,graphicx,pgfplots}

\usepackage[export]{adjustbox}
\usepackage[titletoc,toc,title]{appendix}

\usepackage{algorithm}
\usepackage{mathtools}

\usepackage{afterpage}
\usepackage{comment}
\usepackage{cases}

\newtheorem{theorem}{Theorem}[section]

\newtheorem{lemma}[theorem]{Lemma}                                                                                                                                                                                                                                                                             

\newtheorem{definition}{Definition}

\newtheorem{example}{Example}

\theoremstyle{definition}

\def\RR{\mathbb{R}}

\def\CC{\mathbb{C}}

\def\bu{\mathbf{u}}

\def\:={\coloneqq}

\def\bvp{{\bm{\varphi}}}

\def\H{\mathbf{H}}
\def\p{\partial}

\def\<{\left\langle}
\def\>{\right\rangle}

\def\rmi {\mathrm{i}}

\newcommand{\Om}{\Omega}

\newcommand{\eqnref}[1]{(\ref {#1})}

\def\beq{\begin{equation}}
\def\eeq{\end{equation}}

\newcommand{\ds}{\displaystyle}

\newcommand{\bmat}[1]{\begin{bmatrix}
#1
\end{bmatrix}}

\renewcommand{\Re}{\mbox{Re}}
\renewcommand{\Im}{\mbox{Im}}

\numberwithin{equation}{section}
\numberwithin{figure}{section}

\begin{document}
\title{Analytic asymptotic formulas for effective parameters of planar elastic composites \thanks{\footnotesize This study was supported by the National Research Foundation of Korea (NRF) grant funded by the Korean government (MSIT) (NRF-2021R1A2C1011804).}}

\author{
Daehee Cho\thanks{Department of Mathematical Sciences, Korea Advanced Institute of Science and Technology, 291 Daehak-ro, Yuseong-gu, Daejeon 34141, Republic of Korea (daehee.cho@kaist.ac.kr, mklim@kaist.ac.kr).}\and
Doosung Choi\thanks{ Department of Mathematics, Louisiana State University, Baton Rouge, LA 70803, USA ({dchoi@lsu.edu}).}\and
Mikyoung Lim\footnotemark[2]}

\date{\today}
\maketitle

\begin{abstract}
We investigate the effective elastic properties of periodic dilute two-phase composites consisting of an homogeneous isotropic matrix and a periodic array of rigid inclusions. We assume the rigid inclusion in a unit cell is a simply connected, bounded domain so that there exists an exterior conformal mapping corresponding the inclusion. Recently, an analytical series solution method for the elastic problem with a rigid inclusion was developed based on the layer potential technique and the geometric function theory \cite{Mattei:2021:EAS}.
In this paper, by using the series solution method, we derive expression formulas for the elastic moment tensors--the coefficients of the multipole expansion associated with an elastic inclusion--of an inclusion of arbitrary shape.
These formulas for the elastic moment tensors lead us to analytic asymptotic formulas for the effective parameters of the periodic elastic composites with rigid inclusions in terms of the associated exterior conformal mapping. 
\end{abstract}

%
%


\section{Introduction}
We consider the elastic inclusion problem in two dimensions where a periodic array of inclusions are embedded in an homogeneous isotropic matrix, where the elastic parameters of the matrix and the inclusions are known. 
The main contribution of this paper is the derivation of the analytic formulas for the elastic moment tensors and the effective parameters for an inclusion or a periodic array of rigid inclusions.

Let us consider an elastic inclusion problem where an isotropic elastic inclusion $D$ is embedded in an unbounded isotropic elastic medium. We assume the background medium has the Lam\'{e} constants $({\lambda},{\mu})$ satisfying $\mu>0$ and $ \lambda+\mu>0$. Let $D$ be a bounded, simply connected domain in $\RR^2$ with analytic boundary. 
We denote by $C$ the elastic tensors for the background, that is,
\begin{align*}
C
&:= \big(\lambda \delta_{ij}\delta_{kl} + \mu(\delta_{ik}\delta_{jl} + \delta_{il}\delta_{jk})\big)\left(e_i\otimes e_j\right)\otimes \left(e_k\otimes e_l\right),
\end{align*}
adopting the Einstein summation convention and the tensor product of the standard basis $\{e_j\}_{j=1,2}$ for $\RR^2$.
For a displacement function $\bu$, we define the Lame operator by
\begin{align*}
\mathcal{L}_{\lambda,\mu} \mathbf{u}:= \nabla\cdot(C\widehat{\nabla}\bu)=\mu\Delta\mathbf{u} +  (\lambda+\mu)\nabla\nabla\cdot\mathbf{u} 
\end{align*}
with the strain $\widehat{\nabla}\mathbf{u} = \tfrac{1}{2}(\nabla \mathbf{u} + \nabla \mathbf{u}^T)$. 
The corresponding conormal derivative on $\p D$, that represents the traction force, is
\begin{align*}
\dfrac{\partial \mathbf{u}}{\partial \nu} 
:= (C\widehat{\nabla}\mathbf{u})N
= \lambda (\nabla \cdot \mathbf{u})N + \mu(\nabla\mathbf{u}+\nabla\mathbf{u}^T)N,
\end{align*}
where $N$ is the outward unit normal to $\p D$. 
Suppose that the displacement field $\bu$ is assigned by the far-field loading $\H$ with no body forces. 
If the elastic body were homogeneous without any inclusion, we would have $\bu$ equal to the given far-field loading. The presence of the inclusion induces the perturbation in the displacement field.

For the instance that the inclusion $D$ is rigid, the displacement field $\bu$ is the solution of the system 
\begin{equation} \label{Lame_trans}
\ \left \{
\begin{array} {ll}
\ds \mathcal{L}_{\lambda,\mu} \mathbf{u}= \mathbf{0} \quad &\mbox{ in }  \mathbb{R}^2 \setminus \overline{D},\\[2mm]
\ds \mathbf{u}\big|^+ = b_1\mathbf{R}_1+b_2\mathbf{R}_2+b_3\mathbf{R}_3 \quad &\mbox{ on } \p D,\\[2mm]
\ds \mathbf{u}(x)-\H(x) = O(|x|^{-1}) \quad& \mbox{ as } |x| \rightarrow \infty,
\end{array}
\right.
\end{equation}
where $\bu|^+$ means the limit of $\bu$ from the exterior of $D$ and $\mathbf{R}_j$ are rigid displacements given by
$\mathbf{R}_1(x)=(1,0)$, $\mathbf{R}_2(x)=(0,1)$ and $\mathbf{R}_3(x)=(x_2, -x_1)$ for $x=(x_1,x_2)$. 
The real coefficients $b_j$ are determined to satisfy 
\beq\label{c_j:cond}
\int_{\p D}\dfrac{\partial \mathbf{u}}{\partial \nu}\Big|^+\cdot \mathbf{R}_j\,d\sigma=0\quad\mbox{for all }j=1,2,3.
\eeq
One can find the solution to \eqnref{Lame_trans} using the layer potential operators associated with the Kelvin matrix, denoted by $\Gamma$, the fundamental solution to the Lam\'{e} system. The layer potential solution leads to the multipole expansion of the solution to the elastic inclusion problem: for a given background polynomial solution $\H$, we have
\beq\label{multipole:def}
\bu(x)=\H(x)+\sum_{j=1}^2 \sum_{|\alpha|\geq 1}\sum_{|\beta|\geq 1}
\frac{(-1)^{|\beta|}}{\alpha!\, \beta!}\, \p^\alpha H_j({0})\, \p^\beta\Gamma(x)\,M_{\alpha\beta}^j,\quad|x|\gg1,
\eeq
$M^j_{\alpha\beta}=(m^j_{\alpha\beta 1},m^j_{\alpha \beta 2})$ being the so-called elastic moment tensors (EMTs), $\bu$ the solution to the inclusion problem. Here, $\alpha$, $\beta$ are with multi-indices. 
The EMTs depends on the elastic tensors of the background and inclusion and the geometric features, namely the shape and location, of the inclusion. 
Given its importance in relating the measurement data and the geometric and material variables of the inclusion, various aspects of the EMTs have been investigated, especially for the inverse problem of recovering elastic inclusions and the effective properties of the periodic composites \cite{Ammari:2002:CAE,Kang:2003:IEI,Lim:2011:RSI}; for more results, we refer readers to the book \cite{Ammari:2007:PMT} and references therein.

Effective properties of composite materials have been studied by many authors. Several comprehensive references on this topic include \cite{Garboczi:1996:ICO,Jeffrey:1973:CTR,Jikov:1994:HDO:book,Sangani:1990:CNC,Bieberbach:1916:KDP:book}. 
The literature most closely related to our result is \cite{Ammari:2006:ePE}. 
For the instance of an elastic inclusion problem in two-phase periodic medium, the solution admits the layer potential formulation with the periodic Green's functions are involved  instead of the Kelvin matrix, similarly to the interface problem with an inclusion. In \cite{Ammari:2006:ePE}, the following high-order asymptotic expansion by extending the Maxwell--Garnett formula:
\beq\label{effec:PT}
C^*=C + fM\left(I-fSM\right)^{-1}+O(f^3),
\eeq
$C$ being the elastic tensor in the background medium, $f$ the volume fraction of the inclusion. Here, $M$ denotes the matrix whose entries are the EMTs with $|\alpha|=|\beta|=1$.


For a simple shape such as an disk or an ellipse, the first-order EMTs were expressed \cite{Ammari:2007:BIM,Ammari:2007:PMT}. On the other hand, deriving the asymptotic formulas of the EMTs for inclusions of general shape is challenging due to complexity in the elastic system. In this paper, we employ the complex formulation for the layer potential approach of the Lam\'{e} system to derive analytic expressions for the EMTs of general shape inclusions.

Let $D$ be a planar simply connected, bounded domain. We identify the real two-dimensional plane with the complex plane. From the Riemann mapping theorem, there uniquely exist $\gamma>0$, the conformal radius, and the exterior conformal map--from the exterior of the ball centered at $0$ with the radius $\gamma$ to the exterior of $D$--of the form
\begin{align}\label{def:Psi}
\Phi(w) = w + a_0 + a_1 w^{-1}+a_2 w^{-2}+\cdots,\quad |w|>\gamma,
\end{align}
with complex coefficients $a_n$. 
As a univalent function, $\Phi$ determines the Faber polynomials $F_m(z)$ \cite{Faber:1903:PE}. 
As an essential property, each Faber polynomial admits only one negative order term in $w$, that is,
\begin{align*}
F_m(\Phi(w)) = w^{m}+ c_{m,1}w^{-1}+c_{m,2}w^{-2}+\cdots
\end{align*}
with the so-called Grunsky coefficients $c_{m,n}$. 
As developed in the geometric function theory (see, e.g., the book \cite{Duren:1983:UF}), the Faber polynomials are $m$-th order monic polynomials and form basis for complex analytic functions in the domain. 
The coefficients of the Faber polynomials and the Grunsky coefficients can be explicitly obtained by the conformal mapping coefficients $a_n$.

Recently, an analytical series solution method for the elastic problem with a rigid inclusion was developed by applying the geometric function theory the layer potential approach for the Lam\'{e} system, where the solutions are expressed in terms of the exterior conformal mapping of the inclusion and the applied far-field loading in \cite{Mattei:2021:EAS} (see also \cite{Choi:2021:ASR,Jung:2021:SEL}). This method provides a powerful tool to analyze the explicit dependence of the elastic field perturbation due to the inclusion on the inclusion's shape, in particular, the coefficients of the exterior conformal mapping.

In this paper, building open the results in \cite{Mattei:2021:EAS}, we derive explicit expressions for the EMTs in terms of the exterior conformal mapping coefficients for a bounded, simply connected rigid inclusion with analytic boundary. 
As one of the main results, we obtain the following results where $\mathbb{E}^{(1,i)}$ and $\mathbb{E}^{(2,i)}$ are defined by the expansion coefficients of the series solution  in \cite{Mattei:2021:EAS}:
\begin{theorem}\label{main:thm:intro}
Let $D$ be a bounded, simply connected, and planar Lipschitz domain.  We use \eqref{def:Psi} for the exterior conformal mapping associated with $D$. 
Then, the elastic moment tensors of $D$, $(m^{ij}_{kl})$,  have the form
\begin{align*}
\begin{aligned} 
m^{11}_{11}
&=\dfrac{\mu}{2(\kappa-1)}\left[\Re\left( \mathbb{E}^{(1,1)}\right) + \mathbb{E}^{(2,1)} \right]
- \dfrac{\mu}{2}\left[\Re\left( \mathbb{E}^{(1,2)}\right) + \mathbb{E}^{(2,2)} \right],
\\
m^{22}_{22}
&=\dfrac{\mu}{2(\kappa-1)}\left[ - \Re\left(\mathbb{E}^{(1,1)} \right)  + \mathbb{E}^{(2,1)}\right]
+ \dfrac{\mu}{2}\left[ - \Re\left( \mathbb{E}^{(1,2)}\right) + \mathbb{E}^{(2,2)} \right],
\\
m^{11}_{22}
&=\dfrac{\mu}{2(\kappa-1)}\Re\left(\mathbb{E}^{(2,1)}\right) - \dfrac{\mu}{2}\Re\left( \mathbb{E}^{(1,2)} \right),
\\
m^{12}_{12}
&= \dfrac{\mu}{2}\Im\left(\mathbb{E}^{(1,3)}\right),
\\
m^{12}_{11}
&= \dfrac{\mu}{2}\left[\Re\left( \mathbb{E}^{(1,3)}\right) + \mathbb{E}^{(2.3)} \right],
\\
m^{21}_{22}
&= \dfrac{\mu}{2}\left[-\Re\left(\mathbb{E}^{(1,3)}\right) + \mathbb{E}^{(2,3)}\right],
\end{aligned}
\end{align*}
where $\mathbb{E}^{(1,i)}$ and $\mathbb{E}^{(2,i)}$ are defined by
\begin{align*}
\mathbb{E}^{(1,i)}
:= 2\pi\left(\gamma c_{-1}^{(i)} + \sum_{m=1}^{\infty}a_m\dfrac{c_m^{(i)}}{\gamma^{m}}\right),
\quad
\mathbb{E}^{(2,i)}
:= 2\pi\left(\gamma \overline{c_1^{(i)}} + \sum_{m=1}^{\infty}a_m\dfrac{\overline{c_{-m}^{(i)}}}{\gamma^{m}}\right), \quad  i=1,2,3.
\end{align*}
\end{theorem}
In view of \eqnref{effec:PT}, Theorem \ref{main:thm:intro} provides an explicit asymptotic formula for the effective parameter of  two-phase medium.

%
%

\section{Preliminary}\label{sec:preliminary}

\subsection{Elastic moment tensors}

A fundamental solution $\Gamma = (\Gamma_{ij})_{1\leq i,j\leq2}$ to the Lam\'{e} operator $\mathcal{L}_{\lambda,\mu}$ is the Kelvin matrix $\Gamma$ that is given by
\begin{align*}
\Gamma_{ij}(x) = \frac{\alpha}{2\pi}\delta_{ij}\log|x| - \frac{\beta}{2\pi}\dfrac{x_ix_j}{|x|^2},\quad x\neq 0
\end{align*}
with
\begin{align*}
\alpha = \tfrac{1}{2}\left( \tfrac{1}{\mu} + \tfrac{1}{2\mu+\lambda}\right)
\quad\text{and}\quad
\beta = \tfrac{1}{2}\left( \tfrac{1}{\mu} - \tfrac{1}{2\mu+\lambda}\right).
\end{align*}
In the same manner, we define $\tilde{\alpha}$ and $\tilde{\beta}$ using $\tilde{\lambda}$ and $\tilde{\mu}$.

Let $\Psi$ be the space of rigid displacement, that is,
\begin{align*}
\Psi &:= \left\{ \mathbf{\psi}\in H^1(D)^{2} : \partial_i\psi_j + \partial_j\psi_i=0,~1\leq i,j\leq 2 \right\}\\
&= \text{Span}\left\{\mathbf{R}_j : j=1,2,3\right\}
\end{align*}
with $\mathbf{R}_1(x)=(1,0)$, $\mathbf{R}_2(x)=(0,1)$ and $\mathbf{R}_3(x)=(x_2, -x_1)$ 
and define
\begin{align*}
L^2_{\Psi}(\partial D) := \left\{ f\in L^2(\partial D)^2 : \int_{\partial D}f\cdot \psi d\sigma = 0~\text{for all } \psi\in \Psi\right\}.
\end{align*}


Let us define the single-layer potentials for the operator $\mathcal{L}$ by
\begin{align*}
\mathbf{S}_{\partial D}[{\bm \varphi}](x) &:= \int_{\partial D} \Gamma(x-y){\bm{\varphi}}(y) d\sigma(y) ,\quad x\in\mathbb{R}^{2}
\end{align*}
for $\varphi = (\varphi_1,\varphi_2)\in L^2(\partial D)^{2}$.
It is well-known that the solution to \eqref{Lame_trans} has the form
\begin{align}\label{Scal:density}
\mathbf{u}(x)=
\mathbf{H}(x)  + \mathbf{S}_{\partial D}[{\bm \varphi}](x) ,\quad x\in\mathbb{R}^{2}\setminus \overline{D}
\end{align}
for some density $\bm{\varphi}\in  L^2_\Psi(\partial D)^2$.

\begin{definition} 
 Let $|\alpha|=|\beta|=1$. 
Elastic moment tensors (EMTs) of the rigid inclusion $D$ associated with the background Lam\'{e} constants $(\lambda,\mu)$ are defined by
\begin{align}\label{def:EMT}
M^{j}_{\alpha\beta} = \left( m^{j}_{\alpha \beta 1}, m^{j}_{\alpha \beta 2}\right) := \int_{\partial D}y^{\beta}g^{j}_{\alpha}(y)\,d\sigma(y),
\end{align}
where $g^{j}_{\alpha}\in L_{\Psi}^2(\partial D)^{2} $ be the density function satisfying \eqnref{Scal:density} with 
$\mathbf{H}=x^{\alpha}e_j$. 
\end{definition}
When $|\alpha|=|\beta|=1$, we set $\alpha=e_i$, $\beta=e_k$ $(i,k=1,2)$ and denote
\begin{align*}
m^{ij}_{kl} := m^{j}_{\alpha\beta l},\quad j,l=1,2.
\end{align*}
These first-order EMTs have symmetric properties \cite{Ammari:2002:CAE},
\begin{align}\label{EMT:symmetry}
m^{ij}_{kl} = m^{ji}_{kl} = m^{ij}_{lk} = m^{kl}_{ij}, \quad i,j,k,l=1,2.
\end{align}

%
%

\subsection{Conformal mapping and Faber polynomial}
Let $\Phi:  \{w\in\mathbb{C}: |w|>\gamma\}\to \mathbb{C}\setminus \overline{D} $ be given by \eqnref{def:Psi}.
We then denote scale factors with respect to $\rho$ and $\theta$ as follows:
\begin{align*}
h(\rho,\theta) = \left| \dfrac{\partial \psi}{\partial \rho}\right| = \left|\dfrac{\partial \psi}{\partial \theta}\right|.
\end{align*}
Let us define a density basis, that is, a basis for functions defined on $\partial D$, 
\begin{align}\label{density:basis}
\varphi_{m}(z) = \dfrac{e^{im\theta}}{h},\quad \varphi_{-m}(z) = \dfrac{e^{-im\theta}}{h} \quad \mbox{for } m\in\mathbb{N}.
\end{align}
Notice that the length element on $\partial D$ is given by 
\begin{align*}
d\sigma(\zeta)=h(\rho_0,\theta)d\theta \quad\text{for }\zeta=\psi(\rho_0,\theta).
\end{align*}

The Faber polynomials $F_m$ are determined by the generating relation 
\begin{align*}
\dfrac{w\Phi'(w)}{\Phi(w)-z} = \sum_{m=0}^{\infty}F_m(z)w^{-m},\qquad z\in \overline{D},~|w|>\gamma. 
\end{align*}
The Faber polynomial can be represented by
\beq\label{FaberP}
F_m(z) = \sum_{n=0}^{m} p_{mn} z^n,
\eeq
where $\{ p_{mn} \}_{0\leq n \leq m}$ depends only on $\{ a_n \}_{0 \leq n \leq m-1}$. 
For instance, the first few Faber polynomials has the form
\begin{align*}
F_0(z)&=1,\\[2mm]
F_1(z)&=z-a_0,\\[2mm]
F_2(z)&=z^2-2a_0 z+a_0^2-2a_1.
\end{align*}

\subsection{Explicit series expansion for the single-layer potential in complex form}

We identify $x=(x_1,x_2)\in\RR^2$ with $z=x_1+\rmi x_2\in\CC$. We set
\beq\label{complex:expression}
u(z)=\big((\bu)_1+\rmi (\bu)_2\big)(x),
\eeq
begin $(\cdot)_j$ the $j$-th component of a vector, and, in the same way, define the complex functions $\varphi$ and $H$ corresponding to $\bvp$ and $\H$, respectively. We also define the complex-valued single layer potential as
$$S_{\p \Om}[\varphi](z)=\,\big(({\mathbf{S}}_{\p \Om}[\bm{\varphi}])_1+\rmi ({\mathbf{S}}_{\p \Om}[[\bm{\varphi}])_2\big)(x).$$

For plane elastostatics, the solution of the Lam\'e system has a representation in terms of two holomorphic functions $h$ and $l$ \cite{Muskhelishvili:1953:SBP},
\begin{align*}
2\mu {H}(z) = \kappa h(z) - z\overline{h'(z)} - \overline{l(z)}.
\end{align*}
In particular, as the solution of the Lam\'e system in $\mathbb{C}\setminus \overline{D}$ is the single-layer potential $\mathcal{S}_D[\varphi](z)$, we obtain
\begin{align*}
2\mu\mathcal{S}_D[\varphi](z) = \kappa f(z) - z\overline{f'(z)} - \overline{g(z)}
\end{align*}
for some holomorphic functions $f$ and $g$ expressed as
\begin{align*}
f(z) &= f[\varphi](z) = \beta\mathcal{L}[\varphi](z),\\
g(z) &= g[\varphi](z) = -\alpha \mathcal{L}[\overline{\varphi}](z) - \beta \mathcal{C}[\overline{\zeta}\varphi](z)
\end{align*}
with the complex integral operators,
\begin{align*}
\mathcal{L}[\psi](z) &:= \dfrac{1}{2\pi}\int_{\partial D}\log(z-\zeta)\psi(\zeta)d\sigma(\zeta),\\
\mathcal{C}[\psi](z) &:= \dfrac{\partial }{\partial z}\mathcal{L}[\psi](z) = \dfrac{1}{2\pi}\int_{\partial D}\dfrac{\psi(\zeta)}{z-\zeta}d\sigma(\zeta)
\end{align*}
for any complex function $\psi$ \cite{Ammari:2004:RSI, Ando:2018:SPN}.

We consider the expansion of the density function $\varphi$ on $\partial D$ in terms of the density basis \eqref{density:basis} on $\partial D$:
\begin{align}\label{density}
\varphi(z) = \sum_{m=1}^{\infty}  \left( c_{-m}\varphi_{-m} + c_{m}\varphi_{m} \right)
\end{align}
where $c_n$ are complex numbers.
The constant term, $c_0$ is zero due to $\varphi\in L^2_\Psi(\partial D)$.

As $h$ and $l$ are analytic and Faber polynomials form a basis for analytic functions, we obtain
\begin{align*}
2\mu{H}(z) = \kappa \sum_{m=0}^{\infty}A_m F_m(z) - z\sum_{m=1}^{\infty}\overline{A_m F_m'(z)} - \sum_{m=0}^{\infty}\overline{B_m F_m(z)}
\end{align*}
for some complex sequences $\{A_m\}$ and $\{B_m\}$.


\begin{theorem}[\cite{Mattei:2021:EAS}]
For $z=\psi(w) \in \mathbb{C}\setminus \overline{D}$, the single-layer potential $\mathcal{S}_D[\varphi](z)$ has the following representation from \eqref{density}:
\begin{align}\label{single:rigid:exterior}
2\mathcal{S}_{D}[\varphi](z) = -\alpha v_1 + \beta \psi(w)\overline{v_2(w)} - \beta\overline{v_3(w)},
\end{align}
where
\begin{align*}
v_1 &= \sum_{m=1}^{\infty}\dfrac{1}{m}\left[
c_{-m}(\overline{F_m(z)} - \overline{w^{m}} + w^{-m}) + c_m(F_m-w^{m}+\overline{w^{-m}})
\right],\\
v_2 &= \sum_{m=1}^{\infty}\left[
c_{-m}(\tilde{F}_{-m}(z) - G_{-m}(w)) + c_m(\tilde{F}_{m}(z)-G_m(w))
\right],\\
v_3 &= \sum_{m=1}^{\infty}\left[
c_{-m}\sum_{k=-1}^{\infty}\overline{a_k}(\tilde{F}_{k-m}(z) - G_{k-m}(w)) 
+ c_m\sum_{k=-1}^{\infty}\overline{a_k}(\tilde{F}_{k+m}(z) - G_{k+m}(w))
\right]
\end{align*}
with
\begin{align*}
\tilde{F}_k(z) = \begin{cases}
\dfrac{F_k'(z)}{k}& k \geq 1,\\[2mm]
0&k\leq 0
\end{cases}
\quad \mbox{and} \quad G_k(w):=\dfrac{w^{k-1}}{\psi'(w)}\quad \mbox{for } k\in\mathbb{Z},~|w|>\gamma.
\end{align*}
\end{theorem}


\section{Asymptotic formula for elastic moment tensors}

\subsection{EMTs with general conformal mapping} 
We use certain background solutions as follows 
\begin{align}\label{three:sol}
u^{(1)}(z) = \dfrac{\kappa-1}{2\mu}F_1,\quad
u^{(2)}(z) = -\dfrac{1}{2\mu}\overline{F_1(z)},\quad
u^{(3)}(z) = \dfrac{i}{2\mu}\overline{F_1(z)},
\end{align}
and denote its corresponding exterior density functions $\varphi^{(i)}\in L^2_{\Psi}(\partial D)$ as 
\begin{align}\label{varphi_i}
\varphi^{(i)}  = \sum_{m=1}^{\infty}c_m^{(i)}\dfrac{e^{i m \theta}}{h} + \sum_{m=1}^{\infty}c_{-m}^{(i)}\dfrac{e^{-im\theta}}{h},\qquad i=1,2,3.
\end{align}
that satisfies \eqref{Lame_trans} and \eqref{Scal:density}.

\begin{lemma}\label{lem:E}
Let $D$ be a bounded, simply connected, and planar Lipschitz domain.  We use \eqref{def:Psi} for the exterior conformal mapping associated with $D$. Let us define 
\begin{align}\label{eq:E}
\mathbb{E}^{(1,i)}
:= 2\pi\left(\gamma c_{-1}^{(i)} + \sum_{m=1}^{\infty}a_m\dfrac{c_m^{(i)}}{\gamma^{m}}\right),
\quad
\mathbb{E}^{(2,i)}
:= 2\pi\left(\gamma \overline{c_1^{(i)}} + \sum_{m=1}^{\infty}a_m\dfrac{\overline{c_{-m}^{(i)}}}{\gamma^{m}}\right), \quad  i=1,2,3,
\end{align}
where $c_{m}^{(i)}$ is given in \eqref{varphi_i}.
Then, $\Im\left(\mathbb{E}^{(2,i)}\right)=0$ and
\begin{align*}
\Re\left(\mathbb{E}^{(1,1)}\right) = -(\kappa-1)\mathbb{E}^{(2,2)},\quad
\Re\left(\mathbb{E}^{(1,3)}\right) = -\Im\left(\mathbb{E}^{(1,2)}\right),\quad
\Im\left(\mathbb{E}^{(1,1)}\right) = (\kappa-1)\mathbb{E}^{(2,3)}.
\end{align*}
\end{lemma}
The proof for Lemma \ref{lem:E} is included in those of the following theorem.
Note that $\mathbb{E}^{(i,j)}$ are specific cases of the modified EMTs, $\widetilde{\mathbb{E}}_{nm}^{(t,s)}$ with $n=m=1$ (see Definition 5.1 \cite{Cho:2024:ASR}).

\begin{theorem}\label{theorem:firstEMT}
Let $D$ be a bounded, simply connected, and planar Lipschitz domain.  We use \eqref{def:Psi} for the exterior conformal mapping associated with $D$. 
Then, the elastic moment tensors of $D$, $(m^{ij}_{kl})$,  have the form
\begin{align}\label{first:emt:formula}
\begin{aligned} 
m^{11}_{11}
&=\dfrac{\mu}{2(\kappa-1)}\left[\Re\left( \mathbb{E}^{(1,1)}\right) + \mathbb{E}^{(2,1)} \right]
- \dfrac{\mu}{2}\left[\Re\left( \mathbb{E}^{(1,2)}\right) + \mathbb{E}^{(2,2)} \right],
\\
m^{22}_{22}
&=\dfrac{\mu}{2(\kappa-1)}\left[ - \Re\left(\mathbb{E}^{(1,1)} \right)  + \mathbb{E}^{(2,1)}\right]
+ \dfrac{\mu}{2}\left[ - \Re\left( \mathbb{E}^{(1,2)}\right) + \mathbb{E}^{(2,2)} \right],
\\
m^{11}_{22}
&=\dfrac{\mu}{2(\kappa-1)}\Re\left(\mathbb{E}^{(2,1)}\right) - \dfrac{\mu}{2}\Re\left( \mathbb{E}^{(1,2)} \right),
\\
m^{12}_{12}
&= \dfrac{\mu}{2}\Im\left(\mathbb{E}^{(1,3)}\right),
\\
m^{12}_{11}
&= \dfrac{\mu}{2}\left[\Re\left( \mathbb{E}^{(1,3)}\right) + \mathbb{E}^{(2.3)} \right],
\\
m^{21}_{22}
&= \dfrac{\mu}{2}\left[-\Re\left(\mathbb{E}^{(1,3)}\right) + \mathbb{E}^{(2,3)}\right],
\end{aligned}
\end{align}
where $\mathbb{E}^{(1,i)}$ and $\mathbb{E}^{(2,i)}$ are defined in \eqref{eq:E}.
\end{theorem}
\begin{proof}
Let us recall $u^{(i)}(z)$ and $\varphi^{(i)}$ from \eqref{three:sol} and $\eqref{varphi_i}$.
For convenience, we denote a conformal map as $\Phi(w)=\sum_{k=-1}^{\infty}a_kw^{-k}$ with $a_{-1}=1$.
Then the following identity holds:
\begin{align*}
\begin{aligned}
\int_{\partial D} z\varphi^{(i)} d\sigma 
&= \int_{\partial D} \Phi(w) \sum_{m=1}^{\infty}\left(c_{-m}^{(i)}\varphi_{-m} + c_{m}^{(i)}\varphi_{m}\right)d\sigma
\\
& = \int_{0}^{2\pi} \sum_{k=-1}^{\infty}\left(\gamma^{-k}e^{-ik\theta}a_k\right) 
\sum_{m=1}^{\infty}\left(c_{m}^{(i)}\dfrac{e^{im\theta}}{h} + c_{-m}^{(i)}\dfrac{e^{-im\theta}}{h}\right)h \, d\theta
\\
&= 2\pi\left( \gamma c_{-1}^{(i)} + \sum_{m=1}^{\infty} a_m \dfrac{c_m^{(i)}}{\gamma^{m}}\right) = \mathbb{E}^{(1,i)}.
\end{aligned}
\end{align*}
We can apply the same calculation to $\mathbb{E}^{(2,i)}$ and hence we have
\begin{align}\label{E:reformula}
\mathbb{E}^{(1,i)} = \int_{\partial D}z\varphi^{(i)} d\sigma
\quad\text{and}\quad
\mathbb{E}^{(2,i)} = \int_{\partial D}z\overline{\varphi^{(i)}}d\sigma.
\end{align}
Note that $\varphi\in L^2_{\Psi}(\partial D)$ implies that
\begin{align}\label{E:Faber}
\int_{\partial D}z\varphi^{(i)}d\sigma = \int_{\partial D}(z+c)\varphi^{(i)}d\sigma,\qquad\text{for any }c\in\mathbb{C},
\end{align}
and
\begin{align}\label{ImE2}
\Im \left(\mathbb{E}^{(2,i)}\right)
= \Im \left(\int_{\partial D}\overline{z}\varphi^{(i)}d\sigma\right)
= -\int_{\partial D}\mathbf{R}_3\cdot \varphi^{(i)} d\sigma
=0,
\end{align}
where the rigid motion $\mathbf{R}_3=(x_2,-x_1)\in\Psi$.

We identify the vector $(x,y)\in\mathbb{R}^{2}$ with the complex variable $z=x+iy\in\mathbb{C}$.
Then we have
\begin{align}\label{vector:to:complex}
\begin{aligned}
&\bmat{x\\0} = \dfrac{\mu}{\kappa-1}u^{(1)} - \mu u^{(2)} + \text{Constant},\qquad
&\bmat{0\\y} = \dfrac{\mu}{\kappa-1}u^{(1)} + \mu u^{(2)} + \text{Constant},
\\
&\bmat{0\\x} =-\dfrac{\mu}{\kappa-1}\mathbf{R}_3 + \mu u^{(3)} + \text{Constant},\quad
&\bmat{y\\0} =\dfrac{\mu}{\kappa-1}\mathbf{R}_3 + \mu u^{(3)} + \text{Constant}.
\end{aligned}
\end{align}
By \eqref{E:Faber} and \eqref{ImE2}, and the linearity of the system \eqref{Lame_trans} we can reformulate $g^{j}_{\alpha}$ in \eqref{def:EMT} as 
\begin{align}\label{complex:first:density}
g^{1}_{1} &= \dfrac{\mu}{\kappa-1}\varphi^{(1)} - \mu\varphi^{(2)},\qquad
g^{2}_{2} = \dfrac{\mu}{\kappa-1}\varphi^{(1)} + \mu\varphi^{(2)},\qquad
g^{1}_{2} = g^{2}_{1} = \mu\varphi^{(3)}.
\end{align}
Here, we use the fact that the density functions corresponding rigid motions are identically zero (Lemma 6.14, \cite{Ammari:2004:RSI}).
From the complex-valued formulation, we can rewrite the first order EMTs as
\begin{align}\label{complex:EMT}
m^{ij}_{kl} = \int_{\partial D}y^{k}e_{l}\cdot g^{j}_{i}(y)d\sigma(y) 
= \Re \left(\int_{\partial D} \overline{y^{k} e_{l}} {g^{j}_{i}} d\sigma\right),
\end{align}
where $\overline{y^{k}e_j}$ is a complex conjugate of the complex-valued formula of the vector $y^{k}e_j$.
By using \eqref{ImE2} and \eqref{complex:EMT} we have
\begin{align}\label{EMT:decompose}
\begin{aligned}
m^{ij}_{11}
&=\Re\left( \int_{\partial D}  \dfrac{z+\overline{z}}{2}g^{j}_{i} d\sigma \right)
=\dfrac{1}{2}\Re\left( \int_{\partial D}z g_{i}^{j} + \overline{z}{g_{i}^{j}}d\sigma \right),
\\
m^{ij}_{12} 
&= \Im\left( \int_{\partial D}  \dfrac{z+\overline{z}}{2}g^{j}_{i} d\sigma\right)
=\dfrac{1}{2}\Im\left( \int_{\partial D}z g_{i}^{j}d\sigma\right),
\\
m^{ij}_{21}
&= \Re\left( \int_{\partial D}\dfrac{z-\overline{z}}{2i}g^{j}_{i}d\sigma \right)
= \dfrac{1}{2}\Im\left( \int_{\partial D} z g_{i}^{j} d\sigma\right),
\\
m^{ij}_{22} 
&= \Im\left( \int_{\partial D}  \dfrac{z-\overline{z}}{2i}g^{j}_{i} d\sigma\right) 
= \dfrac{1}{2}\Re\left( \int_{\partial D}-z g_{i}^{j} + \overline{z}{g_{i}^{j}} d\sigma \right).
\end{aligned}
\end{align}
By plugging \eqref{complex:first:density} into \eqref{EMT:decompose}, and apply \eqref{E:reformula} and \eqref{ImE2}, we have
\begin{align*}
m^{11}_{11} 
&= \dfrac{\mu}{2(\kappa-1)}\left[\Re\left( \mathbb{E}^{(1,1)}\right) + \mathbb{E}^{(2,1)} \right]
- \dfrac{\mu}{2}\left[\Re\left( \mathbb{E}^{(1,2)}\right) + \mathbb{E}^{(2,2)} \right],
\\
m^{22}_{22}
&=\dfrac{\mu}{2(\kappa-1)}\left[ - \Re\left(\mathbb{E}^{(1,1)} \right)  + \mathbb{E}^{(2,1)}\right]
+ \dfrac{\mu}{2}\left[ - \Re\left( \mathbb{E}^{(1,2)}\right) + \mathbb{E}^{(2,2)} \right],
\\
m^{12}_{12} 
&= m^{21}_{12}=m^{12}_{21}=m^{21}_{21} = \dfrac{\mu}{2}\Im\left(\mathbb{E}^{(1,3)}\right),
\\
m^{11}_{22}
&=\dfrac{\mu}{2(\kappa-1)}\left[ - \Re\left(\mathbb{E}^{(1,1)}\right)+\mathbb{E}^{(2,1)} \right]
+ \dfrac{\mu}{2}\left[\Re\left(\mathbb{E}^{(1,2)}\right) - \mathbb{E}^{(2,2)} \right],
\\
m^{22}_{11}
&=\dfrac{\mu}{2(\kappa-1)}\left[\Re\left(\mathbb{E}^{(1,1)}\right) + \mathbb{E}^{(2,1)} \right]
+ \dfrac{\mu}{2}\left[\Re\left(\mathbb{E}^{(1,2)}\right) + \mathbb{E}^{(2,2)} \right],
\\
m^{11}_{12} &= m^{11}_{21} = \dfrac{\mu}{2(\kappa-1)}\Im\left( \mathbb{E}^{(1,1)} \right) - \dfrac{\mu}{2}\Im\left( \mathbb{E}^{(1,2)}\right),\\
m^{12}_{11} &= m^{21}_{11} = \dfrac{\mu}{2}\left[\Re\left(\mathbb{E}^{(1,3)}\right) + \mathbb{E}^{(2,3)}\right],
\\
m^{22}_{12} &=m^{22}_{21} = \dfrac{\mu}{2(\kappa-1)}\Im\left( \mathbb{E}^{(1,1)} \right) + \dfrac{\mu}{2}\Im\left( \mathbb{E}^{(1,2)}\right),\\
m^{12}_{22} &=m^{21}_{22} = \dfrac{\mu}{2}\left[-\Re\left(\mathbb{E}^{(1,3)}\right) + \mathbb{E}^{(2,3)}\right].
\end{align*}

In particular, the symmetry of the first order EMTs \eqref{EMT:symmetry} implies that
\begin{align*}
\Re\left(\mathbb{E}^{(1,1)}\right) = -(\kappa-1)\mathbb{E}^{(2,2)},\quad
\Re\left(\mathbb{E}^{(1,3)}\right) = -\Im\left(\mathbb{E}^{(1,2)}\right),\quad
\Im\left(\mathbb{E}^{(1,1)}\right) = (\kappa-1)\mathbb{E}^{(2,3)}.
\end{align*}
Hence, we complete the proof of Lemma \ref{lem:E} and Theorem \ref{theorem:firstEMT}
\end{proof}

\begin{lemma}
Let $\varphi$ be given by \eqref{density} with any given background solution.
Define $\mathbb{E}^{(1)}$ and $\mathbb{E}^{(2)}$ by
\begin{align*}
\mathbb{E}^{(1)} := 2\pi\left(\gamma c_{-1} + \sum_{m=1}^{\infty}a_m\dfrac{c_m}{\gamma^{m}}\right),\qquad
\mathbb{E}^{(2)} := 2\pi\left(\gamma \overline{c_{1}} + \sum_{m=1}^{\infty}a_m\dfrac{\overline{c_{-m}}}{\gamma^{m}}\right).
\end{align*}
Then, the single layer potential corresponding $\varphi$ can be represented with $\mathbb{E}^{(1)}$ and $\mathbb{E}^{(2)}$ by
\begin{align*}
2\mathcal{S}[\varphi](z) = 
-\dfrac{\alpha}{2\pi}\mathbb{E}^{(1)}w^{-1} 
- \dfrac{\alpha}{2\pi}\overline{\mathbb{E}^{(2)}w^{-1}}
+\dfrac{\beta}{2\pi}\overline{\mathbb{E}^{(1)}}w\overline{w^{-2}}
+O(|w|^{-2}).
\end{align*}

\end{lemma}
\begin{proof}
Mattei--Lim present the series expansion of the exterior single layer potential for arbitrary shape of rigid inclusion (Theorem 3.2, \cite{Mattei:2021:EAS}).
This series expansion formula already contains information on the shape of inclusion and its material properties, hence it is not changed even for general case.
We extract the first order terms from \eqref{single:rigid:exterior}, then we have
\begin{align*}
2\mathcal{S}[\varphi](z)
&=-\alpha\sum_{m=1}^{\infty} a_m \dfrac{c_m}{\gamma^{m}} w^{-1} 
-\alpha \sum_{m=1}^{\infty} \overline{a_m} \dfrac{c_{-m}}{\gamma^{m}}\overline{w^{-1}}
-\alpha\gamma c_{-1}w^{-1} - \alpha \gamma c_1 \overline{w^{-1}}\\
&\quad
+\beta \overline{c_{-1}}\gamma w\overline{w^{-2}}
+\beta \sum_{m=1}^{\infty}\overline{a_m}\dfrac{\overline{c_m}}{\gamma^{m}} w\overline{w^{-2}}
+O(|w|^{-2})
\\
&=
-\alpha\left( \gamma c_{-1} + \sum_{m=1}^{\infty} a_m\dfrac{c_m}{\gamma^{m}} \right) w^{-1}
-\alpha\left(\gamma c_1 + \sum_{m=1}^{\infty} \overline{a_m}\dfrac{c_{-m}}{\gamma^{m}}\right)\overline{w^{-1}}\\
&\quad
+\beta\left( \gamma\overline{c_{-1}} + \sum_{m=1}^{\infty}\overline{a_m}\dfrac{\overline{c_m}}{\gamma^{m}}\right)w\overline{w^{-2}}
+O(|w|^{-2}).
\end{align*}
\end{proof}

\section{EMTs with rigid inclusion}
In this section, we give a scheme to attain the quantity of EMTs and present exact values for three examples.
We use the result of Mattei--Lim \cite{Mattei:2021:EAS} to get the coefficients of the density function.

\subsection{Solvability of the series solution method in \cite{Mattei:2021:EAS} for the uniform loading}
Let a conformal map be $\Phi(w)=w+\sum_{m=1}^{N}a_{m}w^{-m}$.
And $d_{m,j},\ j=0,\dots,m-1$ is complex coefficient for the derivative of Faber polynomials, that is,
\begin{align*}
F_m'(z) = \sum_{j=0}^{m-1}d_{m,j}F_j(z).
\end{align*}
We consider the background solution
\begin{align}\label{arb:back:first}
2\mu u(z) = (\kappa A-\overline{A})F_1(z) - a_0\overline{A} - \overline{B F_1(z)},\qquad A,B\in\mathbb{C},
\end{align}
and denote its corresponding density function $\varphi$ as \eqref{density}.
By the result of Mattei--Lim \cite{Mattei:2021:EAS}, the single layer potential is given by
\begin{align*}
2S[\varphi](z)&= -\alpha \dfrac{c_{1}}{\gamma} F_1 
+ \beta\dfrac{\overline{c_1}}{\gamma} F_1
-\alpha \sum_{j=1}^{N} \dfrac{c_{-j}}{j\gamma^{j}}\overline{F_j}\\
&\quad 
-\beta\sum_{j=0}^{N-2}\sum_{k=j+2}^{N}\sum_{m=1}^{k-j-1} \overline{c_{-m}}\dfrac{a_k}{(k-m)\gamma^{2k-m}}\overline{d_{k-m,j}}\overline{F_j(z)}
\\&\quad
-\beta\overline{c_1}\sum_{j=0}^{N}\sum_{k=j}^{N}\dfrac{a_k}{(k+1)\gamma^{2k+1}}\overline{d_{k+1,j}}\overline{F_j(z)}.
\end{align*}

\eqref{Lame_trans} and \eqref{c_j:cond} imply $2\mu(u(z)+S[\varphi](z))=b_1+ib_2+ b_3iz$ on the boundary with $b_1,b_2,b_3\in\mathbb{R}$.
Thus, all Faber polynomial terms should be zero without first order, $F_1(z)$.
By considering $N$-th order Faber polynomial, we have
\begin{align*}
c_{-N} = -\dfrac{\beta}{\alpha}\dfrac{a_{N}N\overline{c_1}}{(N+1)\gamma^{N+1}}\overline{d_{N+1,N}}.
\end{align*}
And $(N-1)$th order Faber polynomial term determines the coefficient $c_{-(N-1)}$.
Inductively, we can ensure that $\{c_{-m}\}$ are represented by $c_1$.
It is easy to check that
\begin{align*}
\Re(c_1) = \dfrac{\gamma}{\beta\mu}\Re(A).
\end{align*}
Now, $\Im\left(\int_{\partial \Omega}z\overline{\varphi(z)}d\sigma(z)\right)=0$ determines the imaginary part of $c_1$.
As $c_1$ is determined by $A$ and $B$ in \eqref{arb:back:first}, we can obtain the coefficients of density function exactly.

\subsection{Inclusions with conformal mappings having upto negative third order}
We provide the explicit formulas for the EMTs for the inclusions with conformal mappings having upto negative third order.
\begin{example}
Let $D$ be a simply connected, planar, and bounded Lipschitz domain.
Assume that the inclusion is a rigid and the conformal map is given by $\Phi(w)=w+a_1/w$.
Then, the first order EMTs are as follows:
\begin{align*}
m^{11}_{11} 
&= \dfrac{2\pi}{\alpha}\Re(a_1) 
+ \dfrac{\pi\gamma^{2}}{\alpha} 
+ \dfrac{\pi}{\kappa-1}\dfrac{\alpha\gamma^{4}-\beta|a_1|^2}{\alpha\beta\gamma^2}  
- \dfrac{\pi}{\alpha}\dfrac{(\alpha + \beta )\gamma^{2}\Im(a_1)^2}{\alpha \gamma^{4}+\beta|a_1|^2},\\
m^{22}_{22}
&= -\dfrac{2\pi}{\alpha}\Re(a_1) 
+ \dfrac{\pi\gamma^{2}}{\alpha} 
+ \dfrac{\pi}{\kappa-1}\dfrac{\alpha\gamma^{4}-\beta|a_1|^2}{\alpha\beta\gamma^2}  
- \dfrac{\pi}{\alpha}\dfrac{(\alpha + \beta )\gamma^{2}\Im(a_1)^2}{\alpha \gamma^{4}+\beta|a_1|^2},\\
m^{11}_{22}
&= \dfrac{\pi}{\kappa-1}\dfrac{\alpha \gamma^{4}-\beta |a_1|^2}{\alpha \beta \gamma^{2}}
- \dfrac{\pi \gamma^{2}}{\alpha }
+ \dfrac{\pi}{\alpha }\dfrac{(\alpha + \beta )\gamma^{2}\Im(a_1)^2}{\alpha \gamma^{4} + \beta |a_1|^2},\\
m^{12}_{12}
&= \dfrac{\pi\gamma^{2}}{\alpha }
- \dfrac{\pi}{\alpha }\dfrac{(\alpha + \beta )\gamma^{2}\Re(a_1)^2}{\alpha \gamma^{4} + \beta |a_1|^2},\\
m^{12}_{11}
&= \dfrac{\pi}{\alpha }\Im(a_1)\left[ 1 + \dfrac{(\alpha + \beta )\gamma^{2}\Re(a_1)}{\alpha \gamma^{4}+\beta|a_1|^2}\right],\\
m^{21}_{22}
&= \dfrac{\pi}{\alpha }\Im(a_1)\left[ 1 - \dfrac{(\alpha + \beta )\gamma^{2}\Re(a_1)}{\alpha \gamma^{4}+\beta|a_1|^2}\right].
\end{align*}
\end{example}

\begin{example}
Let $D$ be a simply connected, planar, and bounded Lipschitz domain.
Assume that the inclusion is a rigid and $\Phi(w)=w+a_1/w+a_2/w^2$.
Then, the first order EMTs are as follows:\begin{align*}
m^{11}_{11} 
&= \dfrac{2\pi}{\alpha}\Re(a_1) 
+ \dfrac{\pi\gamma^{2}}{\alpha} 
+ \dfrac{\pi}{\kappa-1}\dfrac{\alpha\gamma^{6}-\beta\gamma^{2}|a_1|^2-2\beta|a_2|^2}{\alpha\beta\gamma^4}  
- \dfrac{\pi}{\alpha}\dfrac{(\alpha + \beta )\gamma^{4}\Im(a_1)^2}{\alpha\gamma^{6}+\beta\gamma^{2}|a_1|^2+2\beta|a_2|^2},\\
m^{22}_{22}
&= -\dfrac{2\pi}{\alpha}\Re(a_1) 
+ \dfrac{\pi\gamma^{2}}{\alpha} 
+ \dfrac{\pi}{\kappa-1}\dfrac{\alpha\gamma^{6}-\beta\gamma^{2}|a_1|^2-2\beta|a_2|^2}{\alpha\beta\gamma^4}  
- \dfrac{\pi}{\alpha}\dfrac{(\alpha + \beta )\gamma^{4}\Im(a_1)^2}{\alpha\gamma^{6}+\beta\gamma^{2}|a_1|^2+2\beta|a_2|^2},\\
m^{11}_{22}
&= \dfrac{\pi}{\kappa-1}\dfrac{\alpha\gamma^{6}-\beta\gamma^{2}|a_1|^2-2\beta|a_2|^2}{\alpha \beta \gamma^{4}}
- \dfrac{\pi \gamma^{2}}{\alpha }
+ \dfrac{\pi}{\alpha }\dfrac{(\alpha + \beta )\gamma^{4}\Im(a_1)^2}{\alpha\gamma^{6}+\beta\gamma^{2}|a_1|^2+2\beta|a_2|^2},\\
m^{12}_{12}
&= \dfrac{\pi\gamma^{2}}{\alpha }
- \dfrac{\pi}{\alpha }\dfrac{(\alpha + \beta )\gamma^{4}\Re(a_1)^2}{\alpha\gamma^{6}+\beta\gamma^{2}|a_1|^2+2\beta|a_2|^2},\\
m^{12}_{11}
&= \dfrac{\pi}{\alpha }\Im(a_1)\left[ 1 + \dfrac{(\alpha + \beta )\gamma^{4}\Re(a_1)}{\alpha\gamma^{6}+\beta\gamma^{2}|a_1|^2+2\beta|a_2|^2}\right],\\
m^{21}_{22}
&= \dfrac{\pi}{\alpha }\Im(a_1)\left[ 1 - \dfrac{(\alpha + \beta )\gamma^{4}\Re(a_1)}{\alpha\gamma^{6}+\beta\gamma^{2}|a_1|^2+2\beta|a_2|^2}\right].
\end{align*}
\end{example}

\begin{example}
Let $D$ be a simply connected, planar, and bounded Lipschitz domain.
Assume that the inclusion is a rigid and $\Phi(w)=w+a_1/w+a_2/w^2+a_3/w^3$.
Denoting
\begin{align*}
&\mathbb{B}(B) 
=\dfrac{\alpha\gamma^{4}}{\alpha^2\gamma^4\mu-\beta^2\mu|a_3|^2}
 \left[
 -\Im\left(Ba_1-\dfrac{\beta}{\alpha}\overline{B}a_1\overline{a_3}\right)
 +\dfrac{\Im(a_3)}{\gamma^6} + \dfrac{\beta}{\alpha\gamma^2}\Im(a_1^2\overline{a_3}) + \dfrac{\beta}{\alpha\gamma^6}\Im(a_1^2)|a_3|^2
\right]
\\
 &\times\left[
  \gamma + \dfrac{2\beta |a_2|^2}{\alpha \gamma^5} + \dfrac{3\beta|a_3|^2}{\alpha \gamma^7} + \dfrac{\alpha\gamma^4}{\alpha^2\gamma^4\mu-\beta^2\mu|a_3|^2} 
 \left( \dfrac{\beta \mu |a_1|^2}{\gamma^{3}} \left( 1+ \dfrac{\Re(a_3)}{\gamma^4}\right)
 +\dfrac{\beta^2\mu}{\alpha\gamma^3}\left( \Re(a_1^2\overline{a_3}) + \Re(a_1^2)|a_3|^2\right)\right)
 \right]^{-1}\\
\end{align*}
and
\begin{align*}
\mathbb{C}(B) &=
\dfrac{\alpha \gamma^5}{\alpha^2\gamma^4\mu-\beta^2\mu|a_3|^2}
\left[
-\overline{B} - \dfrac{\beta\mu}{\gamma^3}\overline{c_1}a_1 - \dfrac{\beta\mu}{\gamma^7}\overline{c_1}a_1a_3
 + \dfrac{\beta}{\alpha}a_3 \left( B + \dfrac{\beta\mu}{\gamma^3}c_1\overline{a_1} + \dfrac{\beta\mu}{\gamma^7}c_1\overline{a_1}\overline{a_3}\right)
\right],
\end{align*}
then $\mathbb{E}^{(s,t)}$ satisfy
\begin{align*}
\mathbb{E}^{(1,1)} 
&=  2\pi\left( \dfrac{a_1}{\beta\mu} + i \dfrac{a_1}{\gamma}\mathbb{B}(0)+\gamma\mathbb{C}(0)\right),
\\
\mathbb{E}^{(1,2)}
&= 2\pi\left( i\dfrac{a_1}{\gamma}\mathbb{B}(1) + \gamma\mathbb{C}(1)\right),\\
\mathbb{E}^{(1,3)}
&= 2\pi\left( i\dfrac{a_1}{\gamma}\mathbb{B}(i) + \gamma\mathbb{C}(i)\right),\\
\mathbb{E}^{(2,1)}
&= 2\pi \left[
 \left( \dfrac{\gamma^2}{\beta\mu} - \dfrac{2\beta}{\alpha\gamma^5}|a_2|^2 - \dfrac{3\beta}{\alpha\gamma^7}|a_3|^2\right)
- \left( \gamma + \dfrac{2\beta}{\alpha\gamma^5}|a_2|^2 + \dfrac{3\beta}{\alpha\gamma^7}|a_3|^2\right) i\mathbb{B}(0)
+ \dfrac{a_1}{\gamma}\overline{\mathbb{C}(0)}
\right],\\
\mathbb{E}^{(2,2)}
&= 2\pi \left[
- \left( \gamma + \dfrac{2\beta}{\alpha\gamma^5}|a_2|^2 + \dfrac{3\beta}{\alpha\gamma^7}|a_3|^2\right) i\mathbb{B}(1)
+ \dfrac{a_1}{\gamma}\overline{\mathbb{C}(1)}
\right],\\
\mathbb{E}^{(2,3)}
&= 2\pi \left[
- \left( \gamma + \dfrac{2\beta}{\alpha\gamma^5}|a_2|^2 + \dfrac{3\beta}{\alpha\gamma^7}|a_3|^2\right) i\mathbb{B}(i)
+ \dfrac{a_1}{\gamma}\overline{\mathbb{C}(i)}
\right],
\end{align*}
and hence, we can attain the first order EMTs by plugging these formulas into \eqref{first:emt:formula}.
\end{example}

\ifx \bblindex \undefined \def \bblindex #1{} \fi\ifx \bbljournal \undefined
  \def \bbljournal #1{{\em #1}\index{#1@{\em #1}}} \fi\ifx \bblnumber
  \undefined \def \bblnumber #1{{\bf #1}} \fi\ifx \bblvolume \undefined \def
  \bblvolume #1{{\bf #1}} \fi\ifx \noopsort \undefined \def \noopsort #1{} \fi

\end{document}